\documentclass[a4paper]{article}
\usepackage{pdfsync}
\usepackage[utf8]{inputenc}
\usepackage{graphicx}
\usepackage{amssymb,amsfonts,amsmath}
\usepackage{theorem}
\usepackage{hyperref}
\usepackage{graphicx}
\usepackage{multirow}
\usepackage{verbatim}
\usepackage{cite}
\usepackage{color}

\newtheorem{theorem}{Theorem}[section]

\newtheorem{lemma}{Lemma}[section]
\newtheorem{proposition}{Proposition}[section]
{\theorembodyfont{\rmfamily}
\newtheorem{definition}{Definition}[section]
\newtheorem{remark}{Remark}[section]

}
\newenvironment{proof}[1][Proof]{\noindent\textbf{#1.} }{\newline
\hspace*{\textwidth}\hspace*{-0,4cm} \rule{0.5em}{0.5em} \vspace{0,2cm}}

\begin{document}

\title{\textbf{Blow-up collocation solutions of nonlinear homogeneous
    Volterra integral equations}}

\author{R. Ben\'{\i}tez$^{1,*}$, V. J. Bol\'os$^{2,*}$ \\
  \\
  {\small $^1$ Dpto. Matem\'aticas, Centro Universitario de Plasencia, Universidad de Extremadura.}\\
  {\small Avda. Virgen del Puerto 2, 10600 Plasencia, Spain.}\\
  {\small e-mail\textup{: \texttt{rbenitez@unex.es}}} \\
  \\
  {\small $^2$ Dpto. Matem\'aticas para la Econom\'{\i}a y la Empresa, Facultad de Econom\'{\i}a,}\\
  {\small Universidad de Valencia. Avda. Tarongers s/n, 46022 Valencia, Spain.}\\
  {\small e-mail\textup{: \texttt{vicente.bolos@uv.es}}} \\
} \footnotetext{$^*$Work supported by project MTM2008-05460, Spain.}
\date{October 2014}
\maketitle

\begin{abstract}
  In this paper, collocation methods are used for detecting blow-up
  solutions of nonlinear homogeneous Volterra-Hammerstein integral
  equations. To do this, we introduce the concept of ``blow-up
  collocation solution'' and analyze numerically some blow-up time
  estimates using collocation methods in particular examples where
  previous results about existence and uniqueness can be
  applied. Finally, we discuss the relationships between necessary
  conditions for blow-up of collocation solutions and exact solutions.
\end{abstract}

\section{Introduction}

Some engineering and industrial problems are described by explosive
phenomena which are modeled by nonlinear integral equations whose
solutions exhibit blow-up at finite time (see \cite{Kirk09,CalCapo09}
and references therein).  Many authors have studied necessary and
sufficient conditions for the existence of such blow-up
time. Particularly, in \cite{Myd94,Myd99,Okra07,Okra08,Okra10},
equation
\begin{equation}
\label{eq0}
y\left( t\right) =\int _0 ^t K\left( t,s\right) G\left( y\left( s\right) \right) \,
\textrm{d}s , \qquad t\in I:=\left[ 0,T\right]
\end{equation}
was considered. In these works, conditions for the existence of a
finite blow-up time, as well as upper and lower estimates of it, were
given, although they were not very accurate in some cases.

A way for improving these estimations is to study numerical
approximations of the solution; in this aspect, collocation methods
have proven to be a very suitable technique for approximating
nonlinear integral equations, because of its stability and accuracy
(see \cite{Brun04}). Hence, the aim of this paper is to test the
usefulness of collocation methods for detecting blow-up solutions of
the nonlinear homogeneous Volterra-Hammerstein integral
equation (HVHIE) given by equation (\ref{eq0}).

Very recently, Yang and Brunner \cite{YB13} analysed the blow-up
behavior of collocation solutions for a similar Hammerstein-Volterra
integral equation with a convolution kernel
\begin{equation}
\label{eq-1}
y(t) = \phi(t)+\int_0^tK(t-s)G(s,y(s))\textrm{d}s,\qquad t\in I.
\end{equation}

Equation \eqref{eq0} should not be regarded as a particular case of
\eqref{eq-1} (for a convolution kernel), because in \cite{YB13} a
positive non-homogeneity $\phi(t)$ and a Lipschitz-continuous
nonlinearity $G$ were considered. Indeed, under the hypotheses on the
kernel and the nonlinearity considered in \cite{YB13}, it is
well-known that equation \eqref{eq0} has only the trivial solution
($y\equiv 0$) and moreover, under the conditions we shall state in
Section \ref{sec3}, the collocation equations given in \cite{YB13}
lead to the trivial sequence ($y_n =0$ for all
$n\in\mathbb{N}$). Thus, equation \eqref{eq0} considered here is
beyond the scope of the standard techniques, which usually impose
conditions that guarantee the uniqueness of the solutions.

An approach for solving this problem, which we shall follow in this
paper, is writing equation \eqref{eq0} as an implicitly linear
homogeneous Volterra integral equation (HVIE), i.e. setting $z = G\circ y$
and plugging it into the nonlinear integral equation \eqref{eq0} to obtain

\begin{equation}
  \label{eq2}
  z(t) =G\left( \left( \mathcal{V}z\right)
    \left( t\right) \right) =G\left( \int _0 ^t K\left( t,s\right)
    z(s) \, \textrm{d}s \right) , \qquad t\in I,
\end{equation}
being $\mathcal{V}$ the linear \textit{Volterra operator}.  There is a
one-to-one correspondence between solutions of (\ref{eq0}) and
(\ref{eq2}) (see \cite{Brun04,Kras84}), in particular, if $z$ is a
solution of (\ref{eq2}), then $y:=\mathcal{V}z$ is a solution of
(\ref{eq0}).

This paper is structured as follows: in Section \ref{sec2} we introduce
briefly the basic definitions and state the notation. Next, we devote
Section \ref{sec3} to generalize the results about existence and
uniqueness of nontrivial collocation solutions obtained in \cite{BB11}
in order to apply the results to blow-up problems. In Section
\ref{sec4}, we introduce the concept of ``blow-up collocation
solution'' and analyse numerically some blow-up time estimates using
collocation methods in particular examples where the previous results
about existence and uniqueness can be applied. Finally, in Section
\ref{sec5}, we discuss the relationships between necessary conditions
for blow-up of collocation solutions and exact solutions.

\section{Collocation problems for implicitly linear HVIEs}
\label{sec2}

Following the notation of \cite{Brun04}, a \textit{collocation
  problem} for equation (\ref{eq2}) is given by a mesh $I_h:=\left\{
  t_n\, :\, \, 0=t_0<t_1<\ldots <t_N=T\right\} $ and a set of $m$
\textit{collocation parameters} $0\leq c_1<\ldots <c_m\leq 1$.  We
denote $h_n:=t_{n+1}-t_n$ $(n=0,\ldots ,N-1)$ and the quantity
$h:=\textrm{max}\left\{ h_n\, :\, \, 0\leq n\leq N-1\right\} $ is the
\textit{diameter}\footnote{In \cite{Brun04,BB11}, the diameter $h$ is
  also called \textit{stepsize}. Here, in order to emphasize the
  variability of the stepsizes, we shall not follow this notation for
  preventing misunderstandings. Moreover, in the algorithms presented
  in Section \ref{sec:algoritmos}, the initial stepsize $h_0$ is set
  equal to $h$, but in succesive iterations the stepsizes may
  decrease.} of $I_h$. Moreover, the \textit{collocation points} are
given by $t_{n,i} :=t_n+c_ih_n$ $(i=1,\ldots ,m)$, and the set of
collocation points is denoted by $X_h$ (see \cite{Brun04,Brun92}). A
\textit{collocation solution} $z_h$ is then given by the
\textit{collocation equation}
\[
z_h\left( t\right) = G\left( \int _0 ^t K\left( t,s\right) z_h\left( s\right)
  \, \textrm{d}s \right) , \qquad t\in X_h\, ,
\]
where $z_h$ is in the space of piecewise polynomials of degree less
than $m$.

From now on, a ``collocation problem'' or a ``collocation solution''
will be always referred to the implicitly linear HVIE (\ref{eq2}). So,
if we want to obtain an estimation of a solution of the nonlinear
HVHIE (\ref{eq0}), then we have to consider $y_h:=\mathcal{V}z_h$.

At the beginning of Section \ref{sec4}, we define the concepts of
``blow-up collocation problems'' and ``blow-up collocation
solutions''.  To do this, we extend the definition of ``collocation
solution'' to meshes $I_h$ with infinite points. Taking this into
account, it is noteworthy that if $z_h$ is a blow-up collocation
solution, then $y_h$ also blows up at the same blow-up time.

As it is stated in \cite{Brun04}, a collocation solution $z_h$ is
completely determined by the coefficients $Z_{n,i}:=z_h\left(
  t_{n,i}\right) $ $(n=0,\ldots ,N-1)$ $(i=1,\ldots ,m)$, since
$z_h\left( t_n+vh_n\right) =\sum _{j=1}^m L_j\left( v\right) Z_{n,j}$
for all $v\in \left] 0,1\right] $, where $L_1\left( v\right) :=1$,
$L_j\left( v\right) :=\prod _{k\neq j}^m \frac{v-c_k}{c_j-c_k}$
$(j=2,\ldots ,m)$ are the Lagrange fundamental polynomials with
respect to the collocation parameters.  The values of $Z_{n,i}$ are
given by the system
\begin{equation}
\label{eqznigeneral}
Z_{n,i}=G\left( F_n\left( t_{n,i}\right) +h_n\sum _{j=1}^m
B_n\left( i,j\right) Z_{n,j}\right) ,
\end{equation}
where
\[
B_n\left( i,j\right) := \int _0^{c_i}K\left( t_{n,i},t_n+sh_n\right)
L_j\left( s\right) \, \textrm{d}s
\]
and
\begin{equation}
\label{eqlag0}
F_n\left( t\right) :=\int _0^{t_n} K\left( t,s\right)
z_h\left( s\right) \, \textrm{d}s.
\end{equation}
The term $F_n\left( t_{n,i}\right) $ is called the \textit{lag term}.

\section{Existence and uniqueness of nontrivial collocation solutions}
\label{sec3}

In \cite{BB11}, we used collocation methods for approximating the
nontrivial solutions of (\ref{eq0}). In particular, we studied the
collocation solutions of the implicitly linear HVIE \eqref{eq2}, and
assuming that some \textit{general conditions} were
held. Specifically, it was required that the kernel $K$ was a locally
bounded function; nevertheless, non locally bounded kernels,
e.g. weakly singular kernels, appear in many cases of blow-up
solutions (see \cite{Okra10}) and so, in the present work, we shall
first replace this condition with another not excluding this kind of
kernels. Hence, the \textit{general conditions} that we are going to
impose (even if they are not explicitly mentioned) are:
\begin{itemize}
\item\textbf{Over $K$.} The \textit{kernel} $K:\mathbb{R}^2\to \left[
    0,+\infty \right[ $ has its support in $\left\{ \left( t,s\right)
    \in \mathbb{R}^2\, :\, 0\leq s\leq t\right\} $.

  For every $t>0$, the map $s\mapsto K\left( t,s\right) $ is locally
  integrable, and $\mathcal{K}(t):=\int _0 ^t K\left( t,s\right)
  \textrm{d}s$ is a strictly increasing function. Moreover, $\lim
  _{t\rightarrow a^+} \int _a^t K\left( t,s\right) \, \textrm{d}s=0$
  for all $a\geq 0$.  \footnote{This last condition is new with
    respect to those given in \cite{BB11}, and replaces the locally
    boundedness condition.} Note that for convolution kernels (i.e.
  $K(t,s)=k(t-s)$) this last condition is always held, because $k$ is
  locally integrable and then
  \[
  \lim _{t\rightarrow a^+}\int _a^t k(t-s)\, \textrm{d}s = \lim
  _{t\rightarrow a^+}\int _0^{t-a} k(s)\, \textrm{d}s = \lim
  _{t\rightarrow 0^+}\int _0^t k(s)\, \textrm{d}s=0.
  \]
\item \textbf{Over $G$.} The \textit{nonlinearity} $G:\left[ 0,+\infty
  \right[ \to \left[ 0,+\infty \right[ $ is a continuous, strictly
  increasing function, and $G(0)=0$.
\end{itemize}

Note that, since $G$ is injective, the solution $y$ of equation
\eqref{eq0} is also given by $y=G^{-1}\circ z$, where $z$ is a solution of
\eqref{eq2}. Also, since $G(0)=0$, the zero function is always a
solution of both equations \eqref{eq0} and \eqref{eq2}, called
\textit{trivial solution}. Moreover, for convolution kernels, given a
solution $z(t)$ of \eqref{eq2}, any horizontal translation of $z$,
$z_c(t)$ defined by
\[
z_c(t) =
\begin{cases}
  0,& \text{if } t<c\\
  z(t-c)& \text{if } t\geq c,\\
\end{cases}
\]
is also a solution of \eqref{eq2} (see
\cite{AB01,AB03,BB11}). Motivated by this fact we give the next
definitions, that are also valid for nonconvolution kernels:

\begin{definition}\label{def:nearzero}
  We say that a property $\mathcal{P}$ holds \textit{near zero} if
  there exists $\epsilon >0$ such that $\mathcal{P}$ holds on $\left]
    0,\delta \right[ $ for all $0<\delta <\epsilon $.  On the other
  hand, we say that $\mathcal{P}$ holds \textit{away from zero} if
  there exists $\tau >0$ such that $\mathcal{P}$ holds on $\left]
    t,+\infty \right[ $ for all $t>\tau $.

  The concept ``away from zero'' was originally defined in a more
  restrictive way in \cite{BB11}; nevertheless both definitions are
  appropriate for our purposes.
\end{definition}

\begin{definition}
\label{nontrivial}
We say that a solution of (\ref{eq2}) is \textit{nontrivial} if it is
not identically zero near zero.  Moreover, given a collocation
problem, we say that a collocation solution is \textit{nontrivial} if
it is not identically zero in $\left] t_0,t_1\right]$.
\end{definition}

Given a kernel $K$, a nonlinearity $G$ and some collocation parameters
$\left\{ c_1,\ldots ,c_m\right\} $, in \cite{BB11} there were defined
three kinds of existence of nontrivial collocation solutions (of the
corresponding collocation problem) in an interval $I=\left[ 0,T\right]
$ using a mesh $I_h$: \textit{existence near zero}, \textit{existence
  for fine meshes}, and \textit{unconditional existence}. For these
two last kinds of existence, it is ensured the existence of nontrivial
collocation solutions in any interval, and hence, there is no
blow-up. On the other hand, ``existence near zero''
  is equivalent to the existence of a nontrivial \textit{collocation solution
  with adaptive stepsize}, as it is defined in \cite{YB13}. In this
  case, collocation solutions can always be extended a little more,
  but it is not ensured the existence of nontrivial collocation
  solutions for arbitrarily large $T$ and, thus, it can blow up in
  finite time.

  As in \cite{BB11}, we shall restrict our analysis to two particular
  cases of collocation problems, namely:
\begin{itemize}
\item Case 1: $m=1$ with $c_1>0$.
\item Case 2: $m=2$ with $c_1=0$.
\end{itemize}
In these cases, system (\ref{eqznigeneral}) is reduced to a single
nonlinear equation, whose solution is given by the fixed points of
$G\left( \alpha +\beta y\right) $ for some $\alpha,
\beta$. Determining conditions for the existence of collocation
solutions in the general case is a problem of an overwhelming
difficulty. In fact, in \cite{BB11} examples of equations for which,
in cases 1 and 2, there existed nontrivial collocation solutions, yet
in other some other cases, for the same equation, there were no
nontrivial collocation solutions, were given.

The proofs of the results obtained in \cite{BB11} rely on the locally
boundedness of the kernel $K$. Therefore, we will revise such proofs,
imposing the new \textit{general conditions} given at the beginning
of this section, and avoiding other hypotheses that turn out to be
unnecessary. For example, we extend these results to decreasing
convolution kernels, among others.

As in the previous work \cite{BB11}, we will need a Lemma on
nonlinearities $G$ satisfying the \textit{general conditions}:

\begin{lemma}
  \label{lema1} The following statements are equivalent to the
  statement that $\frac{G\left( y\right) }{y}$ is unbounded (in
  $\left] 0,+\infty \right[ $):
\begin{itemize}
\item[(i)] There exists $\beta _0>0$ such that $G\left( \beta y\right)
  $ has nonzero fixed points for all $0<\beta \leq \beta _0$.
\item[(ii)] Given $A\geq 0$, there exists $\beta _A>0$ such that
  $G\left( \alpha +\beta y\right) $ has nonzero fixed points for all
  $0\leq \alpha \leq A$ and for all $0<\beta \leq \beta _A$.
\end{itemize}
\end{lemma}

\subsection{Case 1: $m=1$ with $c_1>0$}

First, we shall consider $m=1$ with $c_1>0$.  Equations
(\ref{eqznigeneral}) are then reduced to
\begin{equation}
\label{eqzn11}
Z_{n,1}=G\left( F_n\left( t_{n,1}\right) +h_n B_n Z_{n,1}\right)
\qquad (n=0,\ldots ,N-1),
\end{equation}
where
\begin{equation}
\label{eqbnc1}
B_n:=B_n\left( 1,1\right) = \int _0^{c_1}K\left( t_{n,1},t_n+sh_n\right) \, \textrm{d}s
\end{equation}
and the lag terms $F_n\left( t_{n,1}\right) $ are given by
(\ref{eqlag0}) with $i=1$. Note that $B_n>0$ because the integrand in
(\ref{eqbnc1}) is strictly positive almost everywhere in $\left]
  0,c_1\right[ $.

\begin{remark}
\label{notahnbn}
In \cite{BB11}, it is ensured that $\lim _{h_n \rightarrow
  0^+}h_nB_n=0$ taking into account (\ref{eqbnc1}) and the old
\textit{general conditions} imposed over $K$. Thus, we must ensure
that the limit also vanishes considering the new \textit{general
  conditions}.  In fact, with a change of variable, we can express
(\ref{eqbnc1}) as
\begin{equation}
\label{newbn}
B_n=\frac{1}{h_n}\int _{t_n}^{t_{n,1}}K\left( t_{n,1},s\right) \, \textrm{d}s.
\end{equation}
Hence, by (\ref{newbn}) and the new condition, we have
\[
\lim _{h_n \rightarrow 0^+}h_nB_n=\lim _{t\rightarrow t_n^+}\int
_{t_n}^t K\left( t,s\right) \, \textrm{d}s=0,
\]
since $t_{n,1}=t_n+c_1h_n$.
\end{remark}

Now, we are in position to give a characterization of the existence
near zero of nontrivial collocation solutions:
\begin{proposition}
\label{prop1}
There is existence near zero
if and only if $\frac{G\left( y\right) }{y}$ is unbounded.
\end{proposition}
\begin{proof}

  ($\Leftarrow $) Let us prove that if $\frac{G\left( y\right) }{y}$
  is unbounded, then there is existence near zero. So, we are going to
  prove by induction over $n$ that there exist $H_n>0$ $\left(
    n=0,\ldots ,N-1\right) $ such that if $0<h_n\leq H_n$ then there
  exist solutions of the system (\ref{eqzn11}) with $Z_{0,1}>0$:

\begin{itemize}

\item For $n=0$, taking into account Remark \ref{notahnbn} and Lemma
  \ref{lema1}-\textit{(i)} (see below), we choose a small enough
  $H_0>0$ such that $0<h_0B_0\leq \beta _0$ for all $0<h_0\leq
  H_0$. So, since the lag term is $0$, we can apply Lemma
  \ref{lema1}-\textit{(i)} to the equation (\ref{eqzn11}), concluding
  that there exist strictly positive solutions for $Z_{0,1}$.

\item Let us suppose that, choosing one of those $Z_{0,1}$ and given
  $n>0$, there exist $H_1,\ldots ,$ $H_{n-1}>0$ such that if
  $0<h_i\leq H_i$ ($i=1,\ldots ,n-1$) then there exist coefficients
  $Z_{1,1},\ldots ,Z_{n-1,1}$ fulfilling the equation
  (\ref{eqzn11}). Note that these coefficients are strictly positive,
  and hence, it is guaranteed that the corresponding collocation
  solution $z_h$ is (strictly) positive in $\left] 0,t_n\right] $.

\item Finally, we are going to prove that there exists $H_n>0$ such
  that if $0<h_n\leq H_n$ then there exists $Z_{n,1}>0$ fulfilling the
  equation (\ref{eqzn11}) with the previous coefficients
  $Z_{0,1},\ldots Z_{n-1,1}$:

  Let us define
  \[
  A:=\max \left\{ F_n\left( t_n+c_1\epsilon \right) \, :\, 0\leq
    \epsilon \leq 1\right\} .
  \]
  Note that $A$ exists because $s\mapsto K\left( t,s\right) $ is
  locally integrable for all $t>0$, and $z_h$ is bounded in $\left[
    0,t_n\right] $; hence, $F_n$ (see (\ref{eqlag0})) is bounded in
  $\left[ t_n,t_n+c_1\right] $.  Moreover, $A\geq 0$ because $K$ and
  $z_h$ are positive functions.

  So, applying Lemma \ref{lema1}-\textit{(ii)} (see below), there
  exists $\beta _A>0$ such that $G\left( \alpha +\beta y\right) $ has
  nonzero (strictly positive) fixed points for all $0\leq \alpha \leq
  A$ and for all $0<\beta \leq \beta _A$. On one hand, taking into
  account Remark \ref{notahnbn}, we choose a small enough $0<H_n\leq
  1$ such that $0<h_nB_n\leq \beta _A$ for all $0<h_n\leq H_n$; on the
  other hand, choosing one of those $h_n$, we have $0\leq F\left(
    t_{n,1}\right) \leq A$. Hence, we obtain the existence of
  $Z_{n,1}$ as the strictly positive fixed point of $G\left( F\left(
      t_{n,1}\right) +h_nB_ny\right) $.

\end{itemize}

($\Rightarrow $) For proving the other condition, we use Lemma
\ref{lema1}-\textit{(i)}, taking into account Remark \ref{notahnbn}.
\end{proof}

A similar result was proved in \cite{BB11} with the additional
hypothesis ``$K\left( t,s\right) \leq K\left( t',s\right) $ for all
$0\leq s\leq t<t'$'' (or ``$k$ is increasing'' for the particular case
of convolution kernels $K(t,s)=k(t-s)$). We have shown that this
hypothesis is not needed and the difference between the proofs is the
choice of $A$.

Taking into account Remark \ref{notahnbn}, we can adapt the results
given in \cite{BB11} about sufficient conditions on existence for fine
meshes and unconditional existence to our new  \textit{general
  conditions}. This will help us to identify collocation problems
without blow-up (see Proposition \ref{prop10} below).

In the following results recall the Definition \ref{def:nearzero} of
the concepts ``near zero'' and ``away from zero''.

\begin{proposition}
\label{prop10}
Let $\frac{G(y)}{y}$ be unbounded near zero.

\begin{itemize}
\item If $K$ is a convolution kernel and $\frac{G(y)}{y}$ is bounded
  away from zero, then there is existence for fine meshes.
\item If $\frac{y}{G(y)}$ is unbounded away from zero, then there is
  unconditional existence.
\end{itemize}

If, in addition, $\frac{G(y)}{y}$ is a strictly decreasing function,
then there is at most one nontrivial collocation solution.
\end{proposition}

The proof is analogous to the one given in \cite{BB11}. Note that the
property ``$\frac{y}{G(y)}$ is unbounded away from zero'' appears
there as ``there exists a sequence $\left\{ y_n\right\}
_{n=1}^{+\infty }$ of positive real numbers and divergent to $+\infty
$ such that $\lim _{n\rightarrow +\infty} \frac{G\left( y_n\right)
}{y_n}=0$'', but both are equivalent.

\subsection{Case 2: $m=2$ with $c_1=0$}

Considering $m=2$ with $c_1=0$, we have to
solve the following equations:
\begin{eqnarray}
\label{eqzn12}
Z_{n,1}&=&G\left( F_n\left( t_{n,1}\right) \right) \\
\label{eqzn2}
Z_{n,2}&=&G\left( F_n\left( t_{n,2}\right) +h_nB_n\left( 2,1\right)
  Z_{n,1}+h_nB_n\left( 2,2\right) Z_{n,2}\right) ,
\end{eqnarray}
for $n=0,\ldots ,N-1$, where
\begin{equation}
\label{eqbn2j}
B_n\left( 2,j\right) = \int _0^{c_2}K\left( t_{n,2},t_n+sh_n\right)
L_j\left( s\right) \, \textrm{d}s,\qquad (j=1,2)
\end{equation}
and $F_n\left( t_{n,i}\right)$ $(i=1,2)$ are given by
(\ref{eqlag0}). Note that $B_n\left( 2,j\right) >0$ for $j=1,2$,
because the integrand in (\ref{eqbn2j}) is strictly positive almost
everywhere in $\left] 0,c_2\right[ $.

\begin{remark}
\label{notahnbn2}
As in the first case, we have to ensure
that $\lim _{h_n\rightarrow 0^+}h_nB_n\left( 2,j\right) =0$ for
$j=1,2$, taking into account the new \textit{general
  conditions} (see Remark \ref{notahnbn}). Actually, since $0<L_j\left( s\right)<1$ for $s\in
\left] 0,c_2\right[$, we have
\begin{equation}
\label{hnbn2j}
0 \leq h_n B_n\left( 2,j\right) \leq h_n\int _0^{c_2}K\left( t_{n,2},t_n+sh_n\right) \,
\textrm{d}s=\int _{t_n}^{t_{n,2}}K\left( t_{n,2},s \right) \, \textrm{d}s.
\end{equation}
Hence, by (\ref{hnbn2j}) and the new condition, the following inequalities are fulfilled:
\[
0\leq \lim _{h_n \rightarrow 0^+}h_nB_n\left( 2,j\right) \leq \lim
_{t\rightarrow t_n^+}\int _{t_n}^t K\left( t,s\right) \,
\textrm{d}s=0.
\]
\end{remark}

Analogously to the previous case, we present a characterization of the
existence near zero of nontrivial collocation solutions:

\begin{proposition}
\label{newprop4}
Let the map $t\mapsto K\left( t,s\right) $ be continuous in $\left]
  s,\epsilon \right[ $ for some $\epsilon >0$ and for all $0\leq
s<\epsilon $ (this hypothesis can be removed if $c_2=1$).  Then there
is existence near zero if and only if $\frac{G\left( y\right) }{y}$ is
unbounded.
\end{proposition}

\begin{proof}

  ($\Leftarrow $) Let us prove that if $\frac{G\left( y\right) }{y}$
  is unbounded, then there is existence near zero. So, we are going to
  prove by induction over $n$ that there exist $H_n>0$ $\left(
    n=0,\ldots ,N-1\right) $ such that if $0<h_n\leq H_n$ then there
  exist solutions of the system (\ref{eqzn2}) with $Z_{0,2}>0$:

\begin{itemize}

\item For $n=0$, taking into account Remark \ref{notahnbn2} and Lemma
  \ref{lema1}-\textit{(i)}, we choose a small enough $H_0>0$ such that
  $0<h_0B_0\left( 2,2\right) \leq \beta _0$ for all $0<h_0\leq
  H_0$. So, since the lag terms are $0$ and $Z_{0,1}=G\left( 0\right)
  =0$, we can apply Lemma \ref{lema1}-\textit{(i)} to the equation
  (\ref{eqzn2}), concluding that there exist strictly positive
  solutions for $Z_{0,2}$.

\item Let us suppose that, choosing one of those $Z_{0,2}$ and given
  $n>0$, there exist constants $H_1,\ldots ,H_{n-1}>0$ such that if $0<h_i\leq
  H_i$ ($i=1,\ldots ,n-1$) then there exist coefficients
  $Z_{1,2},\ldots ,Z_{n-1,2}$ fulfilling the equation
  (\ref{eqzn2}). Moreover, let us suppose that $z_h$ is positive in
  $\left[ 0,t_n\right] $, i.e. these coefficients satisfy $Z_{l,2}\geq
  \left( 1-c_2\right) Z_{l,1}\geq 0$ for $l=1,\ldots ,n-1$, where
  $Z_{l,1}$ is given by (\ref{eqzn12}).

\item Finally, we are going to prove that there exists $H_n>0$ such
  that if $0<h_n\leq H_n$ then there exists $Z_{n,2}>0$ fulfilling the
  equation (\ref{eqzn2}) with the previous coefficients, and $z_h$ is
  positive in $\left[ 0,t_{n+1}\right] $, i.e.  $Z_{n,2}\geq \left(
    1-c_2\right) Z_{n,1}\geq 0$:

    Let us define
    \begin{eqnarray}
    \label{coefa2}
    A &:=& \max \left\{ F_n\left( t_n+c_2\epsilon \right) +\right. \\
    \nonumber
    & & \left. \epsilon  G\left( F_n\left( t_n+c_1\epsilon \right) \right) \int _0^{c_2}K\left( t_{n,2},t_n+sh_n\right)
      L_1\left( s\right) \, \textrm{d}s \,
      :\, 0\leq \epsilon \leq 1\right\} .
    \end{eqnarray}
    Note that $A$ exists and $A\geq 0$ because $s\mapsto K\left(
      t,s\right) $ is locally integrable for all $t>0$, $z_h$ is
    bounded and positive in $\left[ 0,t_n\right] $, the nonlinearity
    $G$ is bounded and positive, and the polynomial $L_1$ is bounded
    and positive in $\left[ 0,c_2\right] $.

    So, applying Lemma \ref{lema1}-\textit{(ii)}, there exists $\beta
    _A>0$ such that $G\left( \alpha +\beta y\right) $ has nonzero
    (strictly positive) fixed points for all $0\leq \alpha \leq A$ and
    for all $0<\beta \leq \beta _A$. On one hand, taking into account
    Remark \ref{notahnbn2}, we choose a small enough $0<H_n\leq 1$
    such that $0<h_nB_n(2,2)\leq \beta _A$ for all $0<h_n\leq H_n$; on
    the other hand, taking into account (\ref{eqlag0}), the lag terms
    $F_n\left( t_{n,i}\right) $ are positive for $i=1,2$, because $K$
    and $z_h$ are positive. Therefore, by (\ref{eqzn12}),
    $Z_{n,1}=G\left( F_n\left( t_{n,1}\right) \right) $ is positive,
    because $G$ is positive.  Moreover, $h_nB_n\left( 2,1\right) $ is
    positive. So,
    \[
    0\leq F_n\left( t_{n,2}\right) +h_nB_n\left( 2,1\right) Z_{n,1}\leq A.
    \]
    Hence, we obtain the existence of $Z_{n,2}$ as the strictly
    positive fixed point of the function $G\left( \alpha +\beta y\right) $ where
    $\alpha :=F_n\left( t_{n,2}\right) +h_nB_n\left( 2,1\right)
    Z_{n,1}$ and $\beta :=h_nB_n(2,2)$.

    Concluding, we have to check that $Z_{n,2}\geq \left( 1-c_2\right)
    Z_{n,1}$. On one hand,
    \[
    Z_{n,2}=G\left( F_n\left( t_{n,2}\right) +h_n\sum_{j=1}^2B_n\left(
        2,j\right) Z_{n,j}\right) \geq G\left( F_n\left(
        t_{n,2}\right) \right)
    \]
    because $B_n\left( 2,j\right) Z_{n,j}\geq 0$ for $j=1,2$.  On the
    other hand, since $t\mapsto K\left( t,s\right) $ is continuous in
    $\left] s,\epsilon \right[ $ for some $\epsilon >0$ and for all
    $0\leq s<\epsilon $, we can suppose that the stepsize is small
    enough for $t_{n,1}<\epsilon $, and then $F_n(t)$ is continuous in
    $t_{n,1}$ (see \ref{eqlag0}). Hence, since $G$ is continuous,
    $G\left( F_n(t)\right) $ is also continuous in $t_{n,1}$,
    i.e. $\lim _{h_n\rightarrow 0^+}G\left( F_n\left( t_{n,2}\right)
    \right) = G\left( F_n\left( t_{n,1}\right) \right) $. Therefore,
    choosing a small enough $h_n$, we have
    \[
    G\left( F_n\left( t_{n,2}\right) \right) \geq \left( 1-c_2\right)
    G\left( F_n\left( t_{n,1}\right) \right) =\left( 1-c_2\right)
    Z_{n,1}.
    \]

\end{itemize}

($\Rightarrow $) For proving the other condition, we use Lemma
\ref{lema1}-\textit{(i)}, taking into account Remark \ref{notahnbn2}.
\end{proof}

It is noteworthy that a similar result was proved in \cite{BB11}, but
the hypothesis on the kernel was ``$K\left( t,s\right) \leq K\left(
  t',s\right) $ for all $0\leq s\leq t<t'$'', and so, unlike
Proposition \ref{newprop4}, it could not be applied to decreasing
convolution kernels (for example). Again, the main difference between
both proofs lies in the choice of $A$.

Note that for convolution kernels $K(t,s)=k(t-s)$, the hypothesis on
$K$ is equivalent to say that $k$ is continuous near zero, which is a
fairly weak hypothesis. On the other hand, for general kernels, this
hypothesis is only needed to ensure that $t\mapsto K\left( t,s\right)
$ is continuous in $t_{n,1}$. Therefore, if it is not continuous, it
is important to choose $h_i$ such that the mapping $t\mapsto K\left(
  t,s\right)$ is continuous at the collocation points $t_{i,1}$
$(i=0,\ldots N-1)$.  More specifically, the hypothesis on $K$ is only
needed to check that $Z_{n,2}\geq \left( 1-c_2\right) Z_{n,1}$, and
hence, to ensure that $z_h$ is positive in $\left[ 0,t_n\right]
$. Nevertheless, this hypothesis is not needed to verify that
$Z_{0,2}>Z_{0,1}=0$, and so, it is always ensured the existence of
$Z_{1,2}$ (and obviously $Z_{1,1}$), even if the hypothesis does not
hold.  Hence, if we use another method to check that $Z_{1,2}\geq
\left( 1-c_2\right) Z_{1,1}$ (e.g., numerically), then it is ensured
the existence of $Z_{2,2}$. Repeating this reasoning we can guarantee
the existence of $Z_{n,2}$, removing the hypothesis on $K$.  Thus, we
can state a result analogous to Proposition \ref{newprop4} without
hypothesis on the kernel:

\begin{proposition}
\label{prop4b}
$\frac{G\left( y\right) }{y}$ is unbounded if and only if there exists
$H_0>0$ such that there are nontrivial collocation solutions in
$\left[ 0,t_1\right] $ for $0<h_0\leq H_0$.  In this case, there
always exists $H_1>0$ such that there are nontrivial collocation
solutions in $\left[ 0,t_2\right] $ for $0<h_1\leq H_1$.

Moreover, if $\frac{G\left( y\right) }{y}$ is unbounded and there is a
\textbf{positive} nontrivial collocation solution in $\left[
  0,t_n\right] $, then there exists $H_n>0$ such that there are
nontrivial collocation solutions in $\left[ 0,t_{n+1}\right] $ for
$0<h_n\leq H_n$.
\end{proposition}

Considering Remark \ref{notahnbn2} and taking the coefficient $A$
given in (\ref{coefa2}), we can adapt some results of \cite{BB11}
about sufficient conditions on existence for fine meshes and
unconditional existence to our new \textit{general conditions}. These
results are useful for identifying collocation problems without
blow-up.

\begin{proposition}
\label{prop11}
Let $\frac{G(y)}{y}$ be unbounded near zero, and let the map $t\mapsto
K\left( t,s\right) $ be continuous in $\left] s,\epsilon \right[ $ for
some $\epsilon >0$ and for all $0\leq s<\epsilon $ (this hypothesis
can be removed if $c_2=1$).

\begin{itemize}
\item If $K$ is a convolution kernel and $\frac{G(y)}{y}$ is bounded
  away from zero, then there is existence for fine meshes.
\item If $\frac{y}{G(y)}$ is unbounded away from zero, then there is
  unconditional existence.
\end{itemize}

If, in addition, $\frac{G(y)}{y}$ is a strictly decreasing function,
then there is at most one nontrivial collocation solution.
\end{proposition}

Moreover, as in Proposition \ref{prop4b}, we can state a result
about unconditional existence without hypothesis on the kernel:

\begin{proposition}
\label{prop6}
Let $\frac{G\left( y\right) }{y}$ be
unbounded near zero.

If $\frac{y}{G(y)}$ is unbounded away from zero and there is a
\textbf{positive} nontrivial collocation solution in $\left[
  0,t_n\right] $, then there is a nontrivial collocation solution in
$\left[ 0,t_{n+1}\right] $ \textbf{for any} $h_n>0$.

If, in addition, $\frac{G\left( y\right) }{y}$ is a strictly
decreasing function, then there is at most one nontrivial collocation
solution.
\end{proposition}

\subsection{Nondivergent existence and uniqueness}

Our interest is the study of existence of nontrivial collocation
solutions using meshes $I_h$ with arbitrarily small $h>0$. So, we are
not interested in collocation problems whose collocation solutions
``escape'' to $+\infty $ when a certain $h_n\rightarrow 0^+$, since
this is a divergence symptom.  Following this criterion, we define the
concept of ``nondivergent existence'':

Let $0=t_0<\ldots <t_n$ be a mesh such that there exist nontrivial
collocation solutions, and let $S_{h_n}$ be the index set of the
nontrivial collocation solutions of the corresponding collocation
problem with mesh $t_0<\ldots <t_n<t_n+h_n$. Given $s\in S_{h_n}$, we
denote by $Z_{s;n,i}$ the coefficients of the corresponding nontrivial
collocation solution verifying equations (\ref{eqznigeneral}). Then,
we say that there is \textit{nondivergent existence in $t_n^+$} if
\[
\inf _{s\in S_{h_n}} \left\{ \max _{i=1,\ldots
    ,m}\left\{ Z_{s;n,i}\right\} \right\}
\]
exists for small enough $h_n>0$ and it does not diverge to $+\infty $
when $h_n\rightarrow 0^+$.  Given a mesh $I_h= \left\{
  0=t_0<\ldots<t_N\right\} $ such that there exist nontrivial
collocation solutions, we say that there is \textit{nondivergent
  existence} if there is nondivergent existence in $t_n^+$ for
$n=0,\ldots, N$.  See \cite{BB11} for a more detailed analysis, where
we also define the concept of \textit{nondivergent uniqueness}.

In \cite{BB11} it is proved the next result:

\begin{proposition}
\label{prop7}
In cases 1 and 2 with existence of nontrivial collocation solutions,
there is nondivergent existence if and only if $\frac{G\left( y\right)
}{y}$ is unbounded near zero.

If, in addition, $G$ is ``well-behaved'', then there is nondivergent
uniqueness.
\end{proposition}

We say that $G$ is ``well-behaved'' if $\frac{G\left( \alpha +y\right)
}{y}$ is strictly decreasing near zero for all $\alpha >0$. Note that
this condition is very weak (see \cite{BB11}).

So, taking into account Propositions \ref{prop1}, \ref{newprop4} and
\ref{prop7}, the main result of this paper is:

\begin{theorem}
\label{teo1}
(Hypothesis only for case 2 with $c_2\neq 1$: the map $t\mapsto
K\left( t,s\right) $ is continuous in $\left] s,\epsilon \right[ $ for
some $\epsilon >0$ and for all $0\leq s<\epsilon $.)

There is nondivergent existence near zero if and only if
$\frac{G\left( y\right) }{y}$ is unbounded near zero.

If, in addition, $G$ is ``well-behaved'' then, there is nondivergent
uniqueness near zero.
\end{theorem}

In the same way, we can combine Propositions \ref{prop4b} and
\ref{prop7} obtaining a result about nondivergent existence and
uniqueness without hypothesis on the kernel, and, as Theorem
\ref{teo1}, it can be useful to study numerically problems with
decreasing convolution kernels in case 2 with $c_2\neq 1$.

We can also combine Propositions \ref{prop10}, \ref{prop11} and
\ref{prop6} with Proposition \ref{prop7}, obtaining results about
nondivergent existence and uniqueness (for fine meshes and
unconditional) that are useful for identifying collocation
problems without blow-up. Some of these results can be reformulated as
necessary conditions for the existence of blow-up (see Section
\ref{sec5}).

\section{Blow-up collocation solutions}
\label{sec4}

In this section we will extend the concept of collocation problem and
collocation solution in order to consider the case of ``blow-up collocation
solutions''.

\begin{definition}
\label{colproblow}
  We say that a collocation problem is a \textit{blow-up collocation
    problem} (or has a blow-up) if the following conditions are held:
  \begin{enumerate}
  \item There exists $T>0$ such that there is no collocation solution
    in $I=\left[ 0,T\right] $ for any mesh $I_h$.
  \item Given $M>0$ there exists $0<\tau<T$, and a
    collocation solution $z_h$ defined on $\left[ 0,\tau \right]$ such
    that $|z_h(t)|> M$ for some $t\in [0,\tau]$.
  \end{enumerate}
\end{definition}
We can not speak about ``blow-up collocation solutions'' in the
classic sense, since ``collocation solutions'' are defined in compact
intervals and obviously they are bounded; so, we have to extend first
the concept of ``collocation solution'' to open intervals $I=\left[
  0,T\right[$ before we are in position to define the notion of ``blow-up
collocation solution''.

\begin{definition}
\label{colblow}
  Let $I:=[0,T[$ and $I_h$ be an infinite mesh given by a strictly
  increasing sequence $\left\{ t_n\right\} _{n=0}^{+\infty }$ with
  $t_0=0$ and convergent to $T$.
  \begin{itemize}
  \item A \textit{collocation solution on $I$} using the mesh $I_h$ is
    a function defined on $I$ such that it is a collocation solution
    (in the classic sense) for any finite submesh $\left\{
      t_n\right\}_{n=0}^N$ with $N\in \mathbb{N}$.
  \item A collocation solution on $I$ is a \textit{blow-up
      collocation solution} (or has a blow-up) with \textit{blow-up
      time} $T$ if it is unbounded.
  \end{itemize}
\end{definition}

\begin{remark}
\label{rem41}
In \cite{YB13} it is defined the concept of \textit{collocation
  solution with adaptive stepsize} that uses an infinite mesh
$I_h$. It is stated that this kind of collocation solutions
\textit{blows up} in finite time if
\[
T_b(I_h):=\lim _{n\rightarrow \infty}t_n=\lim _{n\rightarrow
  \infty}\sum_{i=0}^{n}h_i<\infty ,
\]
without imposing any condition about unboundedness on the collocation
solution. In this case, $T_b(I_h)$ is called the \textit{numerical
  blow-up time}. According to this definition, we could remove the
second point in Definition \ref{colproblow}. However, this definition
of \textit{numerical blow-up time} in \cite{YB13} is made in the
framework of a Lipschitz nonlinearity $G$ in the non-homogenous case,
and it is proved (in case 1) that the corresponding collocation
solution is effectively unbounded if we choose a given adaptive
stepsize that depends on the Lipschitz constant of $G$. Nevertheless,
we can not use this framework in our case (homogeneous case), because
supposing a Lipschitz-continuous nonlinearity $G$ would imply that the
unique solution of \eqref{eq0} or \eqref{eq2} is the trivial one. On
the other hand, if we consider a non-Lipschitz nonlinearity, then we
can not assure that the collocation solution is unbounded since the
adaptive stepsize given in \cite{YB13} becomes zero in the homogeneous
case. Therefore, in Definition \ref{colproblow} it is necessary the
condition about unboundedness, and in Definition \ref{colblow} we
impose explicitly that the collocation solution has to be unbounded.
\end{remark}

Given a collocation problem with nondivergent uniqueness near zero, a
necessary condition for the nondivergent collocation solution to
blow-up is that there is neither existence for fine meshes nor
unconditional existence. So, for example, given a convolution kernel
$K(t,s)=k(t-s)$, in cases 1 and 2 we must require that $\frac{G\left(
    y\right) }{y}$ is unbounded away from zero (moreover, in case 2
with $c_2\neq 1$, we must demand that there exists $\epsilon >0$ such
that $k$ is continuous in $\left] 0,\epsilon \right[ $).

For instance, in \cite{Okra07} it is studied equation (\ref{eq0}) with
convolution kernel $k(t-s)=(t-s)^{\beta }$, $\beta \geq 0$, and
nonlinearity $G(y)=t^{\alpha }$, $0<\alpha <1$, concluding that the
nontrivial (exact) solution does not have blow-up. If we consider a
collocation problem with the above kernel and nonlinearity in cases 1
and 2, then we can ensure unconditional nondivergent uniqueness (by
Propositions \ref{prop10}, \ref{prop11} and \ref{prop7}); thus, we can
also conclude that the nondivergent collocation solution has no
blow-up. Actually, we reach the same conclusion considering any kernel
satisfying the general conditions (in case 1 and case 2 with $c_2=1$)
and such that $t\mapsto K\left( t,s\right) $ is continuous in $\left]
  s,\epsilon \right[ $ for some $\epsilon >0$ and for all $0\leq
s<\epsilon $ (in case 2 with $c_2\neq 1$). In Section \ref{sec5} we
discuss in more detail the relationships between necessary conditions
for blow-up of collocation solutions and exact solutions.

\subsection{Numerical algorithms}
\label{sec:algoritmos}
First we need to recall the definition of \textit{existence near
  zero}, mentioned at the beginning of Section \ref{sec3}, and given
in \cite{BB11}. It is equivalent to the existence of a nontrivial
\textit{collocation solution with adaptive stepsize}, whose definition
is given in \cite{YB13}.

\begin{definition}
  We say that there is \textit{existence near zero} if there exists
  $H_0>0$ such that if $0<h_0\leq H_0$ then there are nontrivial
  collocation solutions in $\left[ 0,t_1\right] $; moreover, there
  exists $H_n>0$ such that if $0<h_n\leq H_n$ then there are
  nontrivial collocation solutions in $\left[ 0,t_{n+1}\right] $ (for
  $n=1,\ldots ,N-1$ and given $h_0,\ldots ,h_{n-1}>0$ such that there
  are nontrivial collocation solutions in $\left[ 0,t_n\right]
  $). Note that, in general, $H_n$ depends on $h_0,\ldots ,h_{n-1}$.
\end{definition}

Given a blow-up collocation problem with nondivergent uniqueness near
zero (always in cases 1 and 2), we are going to describe a general
algorithm to compute the nontrivial collocation solution and estimate
the blow-up time. Given stepsizes $h_0,\ldots ,h_n$,
the collocation solution at $t_{n+1}$ is obtained from the attracting
fixed point of a certain function $y\mapsto G(\alpha+\beta y)$, given
by equation (\ref{eqzn11}) in case 1, or (\ref{eqzn2}) in case
2. Therefore, the key point in the algorithm is to decide whether
there is a fixed point or not and, if so, estimate it.

In order to check that there is no fixed point, the most straightforward
technique consists on iterating the function and, if a certain bound (which may
depend on the fixed points found in the previous steps) is overcomed,
then it is assumed that there is no fixed point.

So, if for a certain $n$ there is no fixed point, then a smaller $h_n$
should be taken (e.g.  $h_n/2$ is used in the examples below), and
this procedure is repeated with the new $\alpha$ and $\beta $
corresponding to this new value of $h_n$. If $h_n$ becomes smaller
than a given tolerance (e.g. $10^{-12}$ in our examples), the
algorithm stops and $t_n=h_0+\ldots +h_{n-1}$ is the estimation of the
blow-up time.

Finally, we need to determine wether the obtained collocation solution
is unbounded or not, but it is not strictly possible since we have
computed the collocation solution only for a finite mesh. Therefore,
we will check if the collocation solution overcomes a previously fixed
threshold bound $M>0$.

To sum up, the algorithm consists mainly on:
\begin{enumerate}
\item Set $n=0$ and an initial stepsize $h_0(=h)$.
\item Check if there is a fixed point in (\ref{eqzn11}) (case 1) or
  (\ref{eqzn2}) (case 2), taking into account Section \ref{sec2}.
\item If there is fixed point, then compute it, set $n=n+1$,
  $h_n=h_{n-1}$, and repeat $2$.
\item If there is not fixed point, then repeat $2$ with a smaller
  stepsize $h_n$. If $h_n$ becomes smaller than a given tolerance, the
  algorithm ends and the blow-up time estimation is given by
  $t_n=h_0+\ldots +h_{n-1}$.
\end{enumerate}

Note that the initial stepsize $h_0$ is also denoted by $h$ because it
coincides with the diameter of the mesh.

\subsection{Examples}

We will consider the following examples:

\begin{enumerate}

\item $K\left( t,s\right) =1$ and $G\left( y\right) =\left\{
\begin{array}{lll}
\sqrt{y} & \, \, \textrm{if} & y\in \left[ 0,1\right] \\
\\
y^2 & \, \, \textrm{if} & y\in \left] 1,+\infty \right[
\end{array}
\right. $

\item $K\left( t,s\right) =1$ and $G\left( y\right) =\left\{
\begin{array}{lll}
\sqrt{y} & \, \, \textrm{if} & y\in \left[ 0,1\right] \\
\\
e^{y-1} & \, \, \textrm{if} & y\in \left] 1,+\infty \right[
\end{array}
\right. $

\item $K\left( t,s\right) =k(t-s)=t-s$ and $G\left( y\right) $ of Example 1.

\item $K\left( t,s\right) =k(t-s)=\frac{1}{\sqrt{\pi \left( t-s\right)
    }}$ and $G\left( y\right) $ of Example 1.

\end{enumerate}

In Example 1, the fixed points can be found analytically:
\begin{itemize}
\item If $0\leq \alpha <1$ and $1-\alpha <\beta \leq \frac{1}{4\alpha }$, then the fixed point
    is $\frac{1-2\alpha \beta -\sqrt{1-4\alpha \beta }}{2\beta ^2}$.
  \item If $0\leq \alpha <1$ and $\beta \leq 1-\alpha $, then the
    fixed point is $\frac{1}{2}\left( \beta +\sqrt{\beta ^2+4\alpha
      }\right) $.
  \item If $\alpha \geq 1$ and $\beta \leq \frac{1}{4\alpha }$, then
    the fixed point is $\frac{1-2\alpha \beta -\sqrt{1-4\alpha \beta
      }}{2\beta ^2}$.
\end{itemize}

If $\beta >\frac{1}{4\alpha }$ then there is no solution for the
corresponding $h_n$. Hence a smaller $h_n$ is taken as described in
the general case. However, both the computational cost and accuracy of
the general and specific algorithms are similar.

In the first two examples, the blow-up time of the corresponding
(exact) solution is $\hat{t}=3$, while in the other two the blow-up
time is unknown. In \cite{Okra10} it was studied a family of equations
to which Example 4 belongs.

In all examples, both $K$ and $G$ fulfill the \textit{general
  conditions} and the hypotheses of Theorem \ref{teo1}, and thus it is
ensured the nondivergent existence and uniqueness near zero. Note that
since the kernel in Example 4 is decreasing and unbounded near zero
this case is out of the scope of the study of \cite{BB11}.

On the other hand, in all examples, $\frac{G(y)}{y}$ is unbounded away
from zero and $\frac{y}{G(y)}$ is bounded away from zero, and hence
there can exist a blow-up (see Propositions \ref{prop10} and
\ref{prop11}).

We have found the numerical nondivergent collocation solutions for a
given diameter $h$ (it is in fact the initial stepsize $h_0$), using
the algorithms described in Section \ref{sec:algoritmos} (the specific
for the Example 1 and the general for the rest). In Figures
\ref{figblowup}, \ref{figblowup2}, \ref{figblowup3} and
\ref{figblowup4} the estimations of the blow-up time of the
collocation solutions for a given initial stepsize $h$ are depicted,
varying $c_1$ in case 1 or $c_2$ in case 2. Note that the different
graphs for different $h$ intersect each other in a fairly good
approximation of the blow-up time, as it is shown in Tables
\ref{tabla1c1}, \ref{tabla1c2}, \ref{tabla2c1} and \ref{tabla2c2}.
Moreover, we can use a more general technique consisting on
extrapolating the minimum ($m_1,m_2$) and maximum ($M_1,M_2$) times of
two results with different initial stepsizes $h_{01},h_{02}$
respectively, with $h_{01}>h_{02}$ (see Figure \ref{fig:extrapol}); in
this way the blow-up time estimation is given by
\begin{equation}
\label{eq:extrapol}
\hat{t}\approx M_1-\frac{(M_1-M_2)(M_1-m_1)}{(M_1-M_2)+(m_2-m_1)}.
\end{equation}

\begin{figure}[tbp]
\begin{center}
  \includegraphics[width=0.25\textwidth]{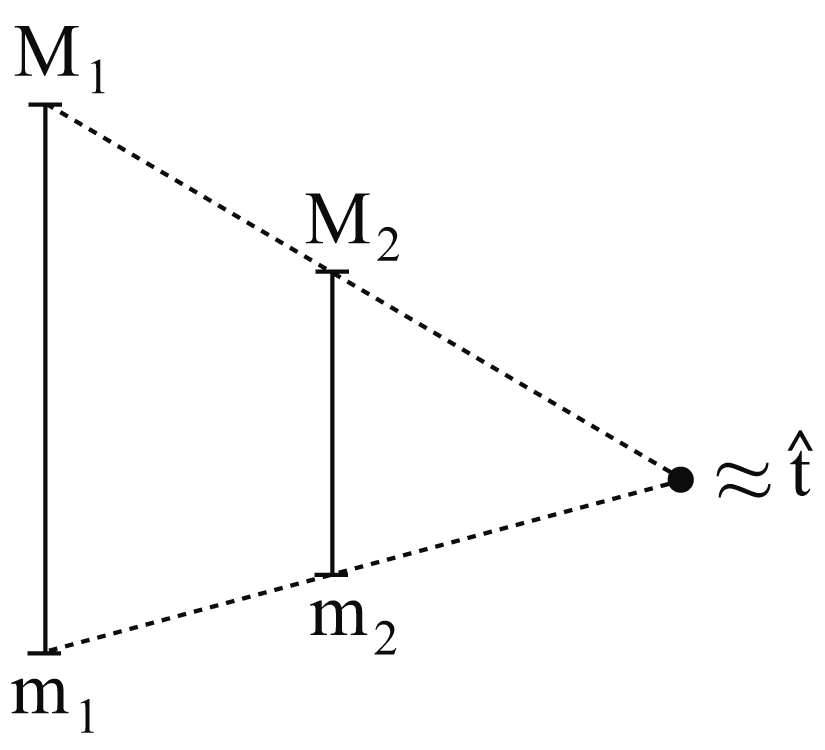}
\end{center}
\caption{Extrapolation technique given by
  \eqref{eq:extrapol}. $M_1,M_2$ and $m_1,m_2$ are the maximum and
  minimum times respectively, using initial stepsizes $h_{01}$ and
  $h_{02}$, with $h_{01}>h_{02}$.}
\label{fig:extrapol}
\end{figure}

In Examples 3 and 4 we do not know the exact value of the blow-up
time. However, in order to make a study of the relative error
analogous to the previous examples, we have taken as blow-up time for
Example 3 the approximation for $h=0.001$ in case 2 with $c_2=0.5$:
$t=5.78482$; in Example 4 the blow-up time value has been taken as the
approximation for $h = 0.0001$ in case 2 with $c_2=2/3$ (\textit{Radau
  I} collocation points): $t=1.645842$. Results are shown in Tables
\ref{tabla3c1}, \ref{tabla3c2}, \ref{tabla4c1} and \ref{tabla4c2}.

The relative error varying $c_1$ (case 1) or $c_2$ (case 2) is the
``relative vertical size'' of the graph, and it decreases at the same
rate as $h$. On the other hand, the relative error of the intersection
decreases faster, in some cases at the same rate as $h^2$.

Moreover, in case 1, the best approximations are obtained with
$c_1\approx 0.5$, and in case 2 with $c_2\approx 2/3$ (approximately
\textit{Radau I}) for Examples 1, 2 and 4 (see Tables
\ref{tabla1c105radau}, \ref{tabla2c105radau} and
\ref{tabla4c105radau}), while for Example 3 the best approximations
are obtained with $c_2\approx 0.5$; however, in Table
\ref{tabla3c105radau} are also shown the approximations and their
corresponding errors for the \textit{Radau I} collocation points.
On the other hand, the intersections technique offers better results,
but at a greater computational cost.

\begin{figure}[tbp]
\begin{center}
  \includegraphics[width=1\textwidth]{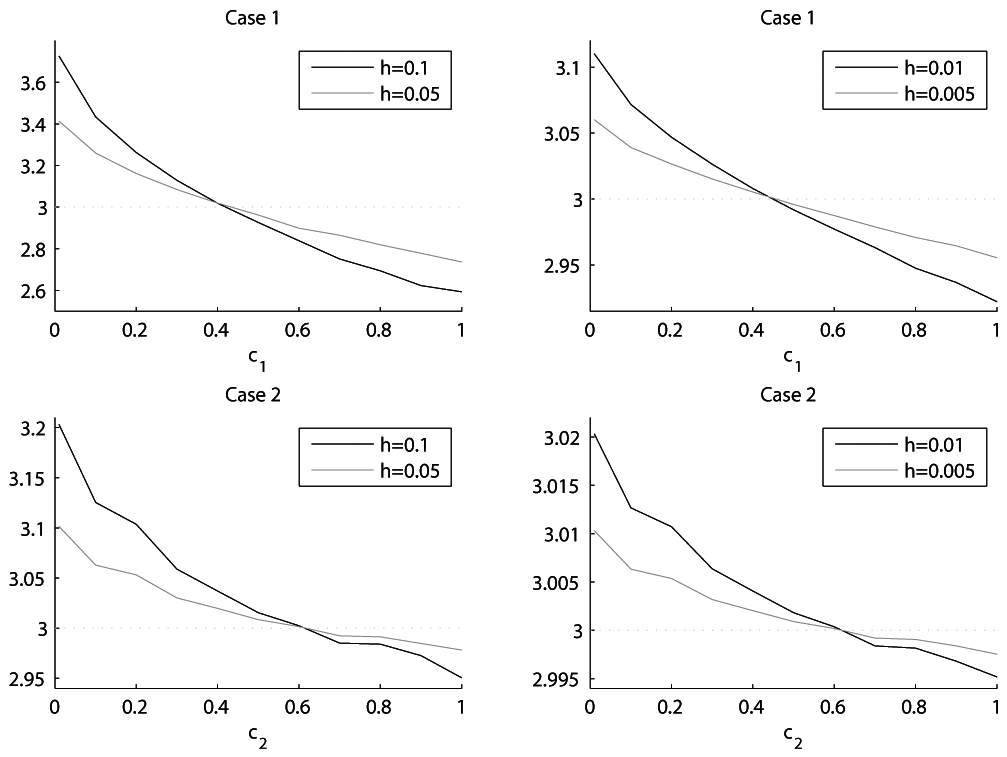}
\end{center}
\caption{Example 1. Numerical estimation of the blow-up time of
  collocation solutions varying $c_1$ (case 1) or $c_2$ (case 2), for
  different initial stepsizes $h$. The blow-up time of the
  corresponding (exact) solution is $\hat{t}=3$. The intersection of
  both curves gives a good approximation of the blow-up time.}
\label{figblowup}
\end{figure}

\begin{figure}[tbp]
\begin{center}
  \includegraphics[width=1\textwidth]{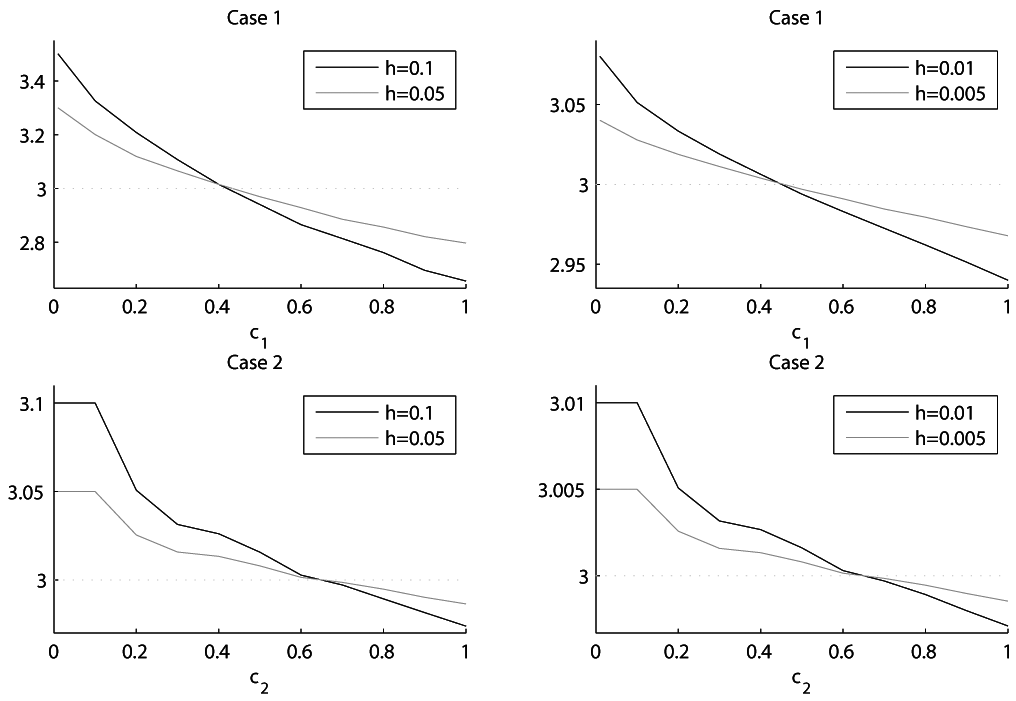}
\end{center}
\caption{Example 2. Numerical estimation of the blow-up time of
  collocation solutions varying $c_1$ (case 1) or $c_2$ (case 2), for
  different inital stepsizes $h$. The blow-up time of the
  corresponding (exact) solution is $\hat{t}=3$. The intersection of
  both curves gives a good approximation of the blow-up time.}
\label{figblowup2}
\end{figure}

\begin{figure}[tbp]
\begin{center}
  \includegraphics[width=1\textwidth]{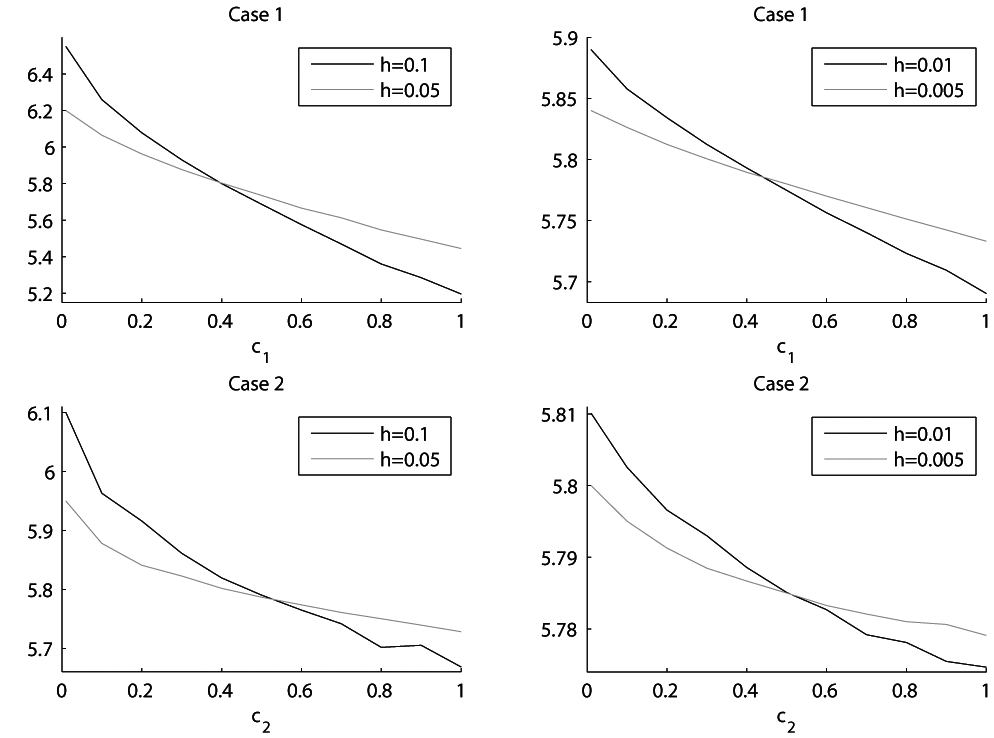}
\end{center}
\caption{Example 3. Numerical estimation of the blow-up time of
  collocation solutions varying $c_1$ (case 1) or $c_2$ (case 2), for
  different initial stepsizes $h$. The intersection of
  both curves gives a good approximation of the blow-up time.}
\label{figblowup3}
\end{figure}

\begin{figure}[tbp]
\begin{center}
  \includegraphics[width=1\textwidth]{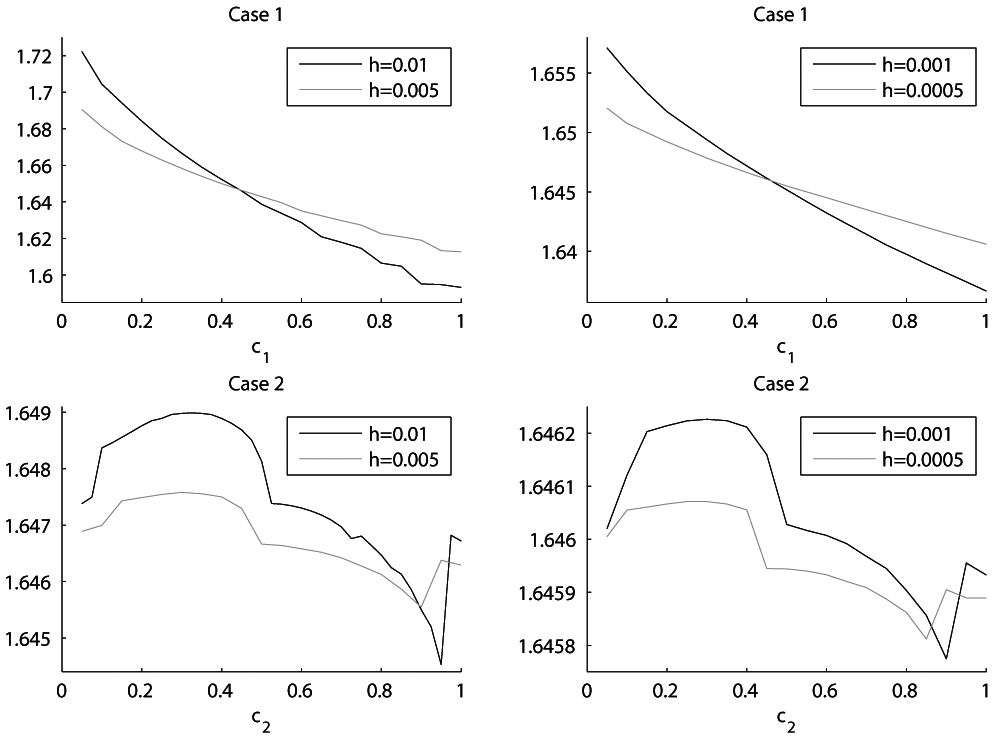}
\end{center}
\caption{Example 4. Numerical estimation of the blow-up time of
  collocation solutions varying $c_1$ (case 1) or $c_2$ (case 2), for
  different initial stepsizes $h$. The intersection of both curves
  gives a good approximation of the blow-up time.}
\label{figblowup4}
\end{figure}

\section{Discussion and comments}
\label{sec5}

The main necessary conditions for a collocation problem to have a
blow-up are obtained from Propositions \ref{prop10} (case 1) and
\ref{prop11} (case 2), and are mostly related to the nonlinearity,
since the assumption on the kernel ``the map $t\mapsto K\left(
  t,s\right) $ is continuous in $\left] s,\epsilon \right[ $ for some
$\epsilon >0$ and for all $0\leq s<\epsilon $'' is only required in
case 2 with $c_2\neq 1$ and it is a very weak hypothesis. Hence,
assuming that the kernel satisfies this hypothesis, the main necessary
condition for the existence of a blow-up is:
\begin{itemize}
\item[1. ] $\frac{y}{G(y)}$ is bounded away from zero.
\end{itemize}
In addition, for convolution kernels, there is another necessary condition:
\begin{itemize}
\item[2. ] $\frac{G(y)}{y}$ is unbounded away from zero, i.e. there
  exists a sequence $\left\{ y_n\right\} _{n=1}^{+\infty }$ of
  positive real numbers and divergent to $+\infty $ such that $\lim
  _{n\rightarrow +\infty} \frac{y_n}{G\left( y_n\right) }=0$.
\end{itemize}

In \cite{Myd99} it is given a necessary and sufficient condition for
the existence of blow-up (exact) solutions for equation (\ref{eq0}) with a
kernel of the form $K(t,s)=(t-s)^{\alpha -1}r(s)$ with $\alpha >0$,
$r$ nondecreasing and continuous for $t\neq 0$, $r(t)=0$ for $t\leq
0$, and $r(t)>0$ for $t>0$:
\begin{equation}
\label{condnecsuf}
\int _{\delta }^{+\infty } \left( \frac{s}{G(s)}\right) ^{1/\alpha }\frac{\textrm{d}s}{s}<+\infty ,\qquad \delta >0.
\end{equation}
This also holds for convolution kernels of Abel type
$K(t,s)=(t-s)^{\alpha -1}$ with $\alpha >0$, generalizing some results
given in \cite{Myd94}. Next we will show that necessary conditions 1
and 2 above mentioned are also necessary conditions for the integral
given in (\ref{condnecsuf}) to be convergent, and thus for the
existence of a blow-up (exact) solution.

\begin{proposition}
\label{propfin1}
If (\ref{condnecsuf}) holds, then $\frac{G(y)}{y}$ is unbounded away
from zero.
\end{proposition}
\begin{proof}
  Let us suppose that $\frac{G(y)}{y}$ is bounded away from zero. So,
  there exists $M>0$ such that $\frac{G(y)}{y}<M$ for all $y>\delta
  $. Hence
\[
\int _{\delta }^{+\infty } \left( \frac{s}{G(s)}\right) ^{1/\alpha
}\frac{\textrm{d}s}{s} > \left( \frac{1}{M}\right) ^{1/\alpha }\int
_{\delta }^{+\infty } \frac{\textrm{d}s}{s}=+\infty .
\]
\end{proof}

\begin{proposition}
  If (\ref{condnecsuf}) holds, then $\frac{y}{G(y)}$ is bounded away
  from zero.
\end{proposition}
\begin{proof}
  By Proposition \ref{propfin1}, $\frac{G(y)}{y}$ is unbounded away
  from zero and hence, there exists a strictly increasing sequence
  $\left\{ y_n\right\} _{n=1}^{+\infty }$ with $y_1>\delta $ and
  divergent to $+\infty $ such that $\frac{y_n}{G\left( y_n\right)
  }<\frac{1}{2^{\alpha }}$ for all $n$.

Let us suppose that $\frac{y}{G(y)}$ is unbounded away from zero; so,
we can choose $\left\{ y_n\right\} _{n=1}^{+\infty }$ such that there
exist $y'_n\in \left] y_n,y_{n+1}\right[ $ with $\frac{y'_n}{G\left(
    y'_n\right) }=1$ for each $n$. Moreover, since $G$ is positive and
strictly increasing, we have that $\frac{y'_n}{G\left( y_n\right)
}>\frac{y'_n}{G\left( y'_n\right) }=1$, and then $y'_n>G\left(
  y_n\right) $. Hence, we have
\begin{eqnarray*}
  \int _{y_n}^{y_{n+1}} \left( \frac{s}{G(s)}\right) ^{1/\alpha }\frac{\textrm{d}s}{s} &>& \int _{y_n}^{y'_n} \left( \frac{s}{G(s)}\right) ^{1/\alpha }\frac{\textrm{d}s}{s} >\int _{y_n}^{y'_n} \left( \frac{s}{G\left( y'_n\right) }\right) ^{1/\alpha }\frac{\textrm{d}s}{s} \\
  &=& \alpha \left( 1-\left( \frac{y_n}{y'_n}\right) ^{1/\alpha }\right) >\alpha \left( 1-\left( \frac{y_n}{G\left( y_n\right) }\right) ^{1/\alpha }\right) >\frac{\alpha }{2} .
\end{eqnarray*}
Therefore, (\ref{condnecsuf}) does not hold.
\end{proof}

The results presented here provide a guide for future research. An
interesting problem which is not fully resolved is to determine the
relationship between the existence of blow-up in exact solutions and
in collocation solutions.

Another open problem is to assure the
unboundedness of a collocation solution with adaptive stepsize (in the
sense of \cite{YB13}) using an infinite mesh $I_h$ with $\lim
_{n\rightarrow \infty}t_n < \infty$ in the general case of a
non-Lipschitz nonlinearity $G$ (see Remark \ref{rem41}).
In other words, if we can not arrive at a given time $T$ with any
collocation solution (i.e. using any mesh), then, is there a
collocation blow-up? If so, and in terms of Definition \ref{colproblow}, it would
be sufficient to hold only the first point for a collocation problem
to have a blow-up.

\newpage

\section*{Tables}

\begin{table}[h]
\begin{center}
\begin{tabular}{|c|c|c|c|c|c|c|c|}
  \hline
  Case 1 & \multicolumn{3}{|c|}{Varying $c_1$} & \multicolumn{2}{|c|}{Extrapolation} & \multicolumn{2}{|c|}{Intersection} \\
  \hline
  $h$ & min. $t$ & max. $t$ & Rel. error & $t$ & Rel. error & $t$ & Rel. error \\
  \hline
  \multirow{2}{*}{$0.1$} & $2.60$ & $3.66$ & \multirow{2}{*}{$4\cdot 10^{-1}$} & \multirow{4}{*}{$3.033$} & \multirow{4}{*}{$1.1\cdot 10^{-2}$} & & \multirow{4}{*}{$10^{-2}$} \\
  & \small{($c_1=1$)} & \small{($c_1=0.01$)} & & & & $3.03$ & \\
  \cline{1-4}
  \multirow{2}{*}{$0.05$} & $2.78$ & $3.40$ & \multirow{2}{*}{$2\cdot 10^{-1}$} & & & \small{($c_1=0.39$)} & \\
  & \small{($c_1=1$)} & \small{($c_1=0.01$)} & & & & & \\
  \hline
  \multirow{2}{*}{$0.01$} & $2.92$ & $3.11$ & \multirow{2}{*}{$6\cdot 10^{-2}$} & \multirow{4}{*}{$3.0044$} & \multirow{4}{*}{$1.5\cdot 10^{-3}$} & & \multirow{4}{*}{$7\cdot 10^{-4}$} \\
  & \small{($c_1=1$)} & \small{($c_1=0.01$)} & & & & $3.002$ & \\
  \cline{1-4}
  \multirow{2}{*}{$0.005$} & $2.96$ & $3.06$ & \multirow{2}{*}{$3\cdot 10^{-2}$} & & & \small{($c_1=0.44$)} & \\
  & \small{($c_1=1$)} & \small{($c_1=0.01$)} & & & & & \\
  \hline
\end{tabular}
\end{center}
\caption{Example 1. Numerical data of Figure \ref{figblowup} (case 1: $m=1$, $c_1>0$).}
\label{tabla1c1}
\end{table}

\begin{table}[h]
\begin{center}
\begin{tabular}{|c|c|c|c|c|c|c|c|}
  \hline
  Case 2 & \multicolumn{3}{|c|}{Varying $c_2$} & \multicolumn{2}{|c|}{Extrapolation} & \multicolumn{2}{|c|}{Intersection} \\
  \hline
  $h$ & min. $t$ & max. $t$ & Rel. error & $t$ & Rel. error & $t$ & Rel. error \\
  \hline
  \multirow{2}{*}{$0.1$} & $2.95$ & $3.20$ & \multirow{2}{*}{$8\cdot 10^{-2}$} & \multirow{4}{*}{$3.008$} & \multirow{4}{*}{$2.6\cdot 10^{-3}$} & & \multirow{4}{*}{$10^{-4}$} \\
  & \small{($c_2=1$)} & \small{($c_2=0.01$)} & & & & $3.0003$ & \\
  \cline{1-4}
  \multirow{2}{*}{$0.05$} & $2.98$ & $3.10$ & \multirow{2}{*}{$4\cdot 10^{-2}$} & & & \small{($c_2=0.618$)} & \\
  & \small{($c_2=1$)} & \small{($c_2=0.01$)} & & & & & \\
  \hline
  \multirow{2}{*}{$0.01$} & $2.995$ & $3.020$ & \multirow{2}{*}{$8\cdot 10^{-3}$} & \multirow{4}{*}{$3.0008$} & \multirow{4}{*}{$2.6\cdot 10^{-4}$} & & \multirow{4}{*}{$10^{-6}$} \\
  & \small{($c_2=1$)} & \small{($c_2=0.01$)} & & & & $3.000003$ & \\
  \cline{1-4}
  \multirow{2}{*}{$0.005$} & $2.998$ & $3.010$ & \multirow{2}{*}{$4\cdot 10^{-3}$} & & & \small{($c_2=0.623$)} & \\
  & \small{($c_2=1$)} & \small{($c_2=0.01$)} & & & & & \\
  \hline
\end{tabular}
\end{center}
\caption{Example 1. Numerical data of Figure \ref{figblowup} (case 2:$m=2$, $c_1=0$).}
\label{tabla1c2}
\end{table}

\begin{table}[h]
\begin{center}
\begin{tabular}{|c|c|c|c|c|}
  \hline
  \multirow{2}{*}{$h$} & \multicolumn{2}{|c|}{Case 1 - $c_1=0.5$} & \multicolumn{2}{|c|}{Case 2 - Radau I} \\
  \cline{2-5}
   & Blow-up & Rel. error & Blow-up & Rel. error \\
  \hline
  $0.1$ & $2.933883$ & $2.2\cdot 10^{-2}$ & $2.995253$ & $1.6\cdot 10^{-3}$ \\
  \hline
  $0.05$ & $2.965002$ & $1.2\cdot 10^{-2}$ & $2.997602$ & $8\cdot 10^{-4}$ \\
  \hline
  $0.01$ & $2.992885$ & $2.4q\cdot 10^{-3}$ & $2.999519$ & $1.6\cdot 10^{-4}$ \\
  \hline
  $0.005$ & $2.996434$ & $1.2\cdot 10^{-3}$ & $2.999759$ & $8\cdot 10^{-5}$ \\
  \hline
\end{tabular}
\end{center}
\caption{Example 1. Numerical estimations and relative errors of the blow-up time of collocation solutions in case 1 with $c_1=0.5$, and case 2 with \textit{Radau I} collocation points, $c_1=0$, $c_2=2/3$, for different initial stepsizes $h$.}
\label{tabla1c105radau}
\end{table}

\begin{table}[h]
\begin{center}
\begin{tabular}{|c|c|c|c|c|c|c|c|}
\hline
Case 1 & \multicolumn{3}{|c|}{Varying $c_1$} & \multicolumn{2}{|c|}{Extrapolation} & \multicolumn{2}{|c|}{Intersection} \\
\hline
$h$ & min. $t$ & max. $t$ & Rel. error & $t$ & Rel. error & $t$ & Rel. error \\
\hline
\multirow{2}{*}{$0.1$} & $2.68$ & $3.50$ & \multirow{2}{*}{$3\cdot 10^{-1}$} & \multirow{4}{*}{$3.018$} & \multirow{4}{*}{$6\cdot 10^{-3}$} & & \multirow{4}{*}{$5\cdot 10^{-3}$} \\
 & \small{($c_1=1$)} & \small{($c_1=0.01$)} & & & & $3.015$ & \\
\cline{1-4}
\multirow{2}{*}{$0.05$} & $2.82$ & $3.30$ & \multirow{2}{*}{$1.5\cdot 10^{-1}$} & & & \small{($c_1=0.399$)} & \\
 & \small{($c_1=1$)} & \small{($c_1=0.01$)} & & & & & \\
\hline
\multirow{2}{*}{$0.01$} & $2.94$ & $3.08$ & \multirow{2}{*}{$4.6\cdot 10^{-2}$} & \multirow{4}{*}{$3.01$} & \multirow{4}{*}{$3.3\cdot 10^{-3}$} & & \multirow{4}{*}{$2.6\cdot 10^{-4}$} \\
 & \small{($c_1=1$)} & \small{($c_1=0.01$)} & & & & $3.0008$ & \\
\cline{1-4}
\multirow{2}{*}{$0.005$} & $2.98$ & $3.04$ & \multirow{2}{*}{$2.3\cdot 10^{-2}$} & & & \small{($c_1=0.444$)} & \\
 & \small{($c_1=1$)} & \small{($c_1=0.01$)} & & & & & \\
\hline
\end{tabular}
\end{center}
\caption{Example 2. Numerical data of Figure \ref{figblowup2} (case 1: $m=1$, $c_1>0$).}
\label{tabla2c1}
\end{table}

\begin{table}[h]
\begin{center}
\begin{tabular}{|c|c|c|c|c|c|c|c|}
\hline
Case 2 & \multicolumn{3}{|c|}{Varying $c_2$} & \multicolumn{2}{|c|}{Extrapolation} & \multicolumn{2}{|c|}{Intersection} \\
\hline
$h$ & min. $t$ & max. $t$ & Rel. error & $t$ & Rel. error & $t$ & Rel. error \\
\hline
\multirow{2}{*}{$0.1$} & $2.97$ & $3.10$ & \multirow{2}{*}{$4.3\cdot 10^{-2}$} & \multirow{4}{*}{$3.007$}& \multirow{4}{*}{$2.4\cdot 10^{-3}$} & & \multirow{4}{*}{$4\cdot 10^{-5}$} \\
 & \small{($c_2=1$)} & \small{($c_2=0.01$)} & & & & $3.00012$ & \\
\cline{1-4}
\multirow{2}{*}{$0.05$} & $2.99$ & $3.05$ & \multirow{2}{*}{$2\cdot 10^{-2}$} & & & \small{($c_2=0.6466$)} & \\
 & \small{($c_2=1$)} & \small{($c_2=0.01$)} & & & & & \\
\hline
\multirow{2}{*}{$0.01$} & $2.997$ & $3.010$ & \multirow{2}{*}{$4.3\cdot 10^{-3}$} & \multirow{4}{*}{$3.000003$} & \multirow{4}{*}{$10^{-6}$} & & \multirow{4}{*}{$3.6\cdot 10^{-6}$} \\
 & \small{($c_2=1$)} & \small{($c_2=0.01$)} & & & & $3.000011$ & \\
\cline{1-4}
\multirow{2}{*}{$0.005$} & $2.9985$ & $3.005$ & \multirow{2}{*}{$2.15\cdot 10^{-3}$} & & & \small{($c_2=0.65114$)} & \\
 & \small{($c_2=1$)} & \small{($c_2=0.01$)} & & & & & \\
\hline
\end{tabular}
\end{center}
\caption{Example 2. Numerical data of Figure \ref{figblowup2} (case 2: $m=2$, $c_1=0$).}
\label{tabla2c2}
\end{table}

\begin{table}[h]
\begin{center}
\begin{tabular}{|c|c|c|c|c|}
  \hline
  \multirow{2}{*}{$h$} & \multicolumn{2}{|c|}{Case 1 - $c_1=0.5$} & \multicolumn{2}{|c|}{Case 2 - Radau I} \\
  \cline{2-5}
   & Blow-up & Rel. error & Blow-up & Rel. error \\
  \hline
  $0.1$ & $2.941795$ & $2\cdot 10^{-2}$ & $2.999559$ & $1.5\cdot 10^{-4}$ \\
  \hline
  $0.05$ & $2.970341$ & $10^{-2}$ & $2.999778$ & $7.4\cdot 10^{-5}$ \\
  \hline
  $0.01$ & $2.993979$ & $2\cdot 10^{-3}$ & $2.999955$ & $1.5\cdot 10^{-5}$ \\
  \hline
  $0.005$ & $2.996983$ & $10^{-3}$ & $2.999978$ & $7.3\cdot 10^{-6}$ \\
  \hline
\end{tabular}
\end{center}
\caption{Example 2. Numerical estimations and relative errors of the blow-up time of collocation solutions in case 1 with $c_1=0.5$, and case 2 with \textit{Radau I} collocation points, $c_1=0$, $c_2=2/3$, for different initial stepsizes $h$.}
\label{tabla2c105radau}
\end{table}

\begin{table}[h]
\begin{center}
\begin{tabular}{|c|c|c|c|c|c|c|c|}
\hline
Case 1 & \multicolumn{3}{|c|}{Varying $c_1$} & \multicolumn{2}{|c|}{Extrapolation} & \multicolumn{2}{|c|}{Intersection} \\
\hline
$h$ & min. $t$ & max. $t$ & Rel. error & $t$ & Rel. error & $t$ & Rel. error \\
\hline
\multirow{2}{*}{$0.1$} & $5.20$ & $6.55$ & \multirow{2}{*}{$2.3\cdot 10^{-1}$} & \multirow{4}{*}{$5.7625$} & \multirow{4}{*}{$3.8\cdot 10^{-3}$} & & \multirow{4}{*}{$3.6\cdot 10^{-3}$} \\
 & \small{($c_1=1$)} & \small{($c_1=0.01$)} & & & & $5.8055$ & \\
\cline{1-4}
\multirow{2}{*}{$0.05$} & $5.45$ & $6.20$ & \multirow{2}{*}{$1.3\cdot 10^{-1}$} & & & \small{($c_1=0.3961$)} & \\
 & \small{($c_1=1$)} & \small{($c_1=0.01$)} & & & & & \\
\hline
\multirow{2}{*}{$0.01$} & $5.71$ & $5.89$ & \multirow{2}{*}{$3.1\cdot 10^{-2}$} & \multirow{4}{*}{$5.77$} & \multirow{4}{*}{$2.6\cdot 10^{-3}$} & & \multirow{4}{*}{$1.4\cdot 10^{-4}$} \\
 & \small{($c_1=1$)} & \small{($c_1=0.01$)} & & & & $5.7856$ & \\
\cline{1-4}
\multirow{2}{*}{$0.005$} & $5.735$ & $5.840$ & \multirow{2}{*}{$1.8\cdot 10^{-2}$} & & & \small{($c_1=0.441$)} & \\
 & \small{($c_1=1$)} & \small{($c_1=0.01$)} & & & & & \\
\hline
\end{tabular}
\end{center}
\caption{Example 3. Numerical data of Figure \ref{figblowup3} (case 1: $m=1$, $c_1>0$). The blow-up time of the corresponding (exact) solution is assumed to be $\hat{t}\approx 5.78482$.}
\label{tabla3c1}
\end{table}

\begin{table}[h]
\begin{center}
\begin{tabular}{|c|c|c|c|c|c|c|c|}
\hline
Case 2 & \multicolumn{3}{|c|}{Varying $c_2$} & \multicolumn{2}{|c|}{Extrapolation} & \multicolumn{2}{|c|}{Intersection} \\
\hline
$h$ & min. $t$ & max. $t$ & Rel. error & $t$ & Rel. error & $t$ & Rel. error \\
\hline
\multirow{2}{*}{$0.1$} & $5.67$ & $6.10$ & \multirow{2}{*}{$7.4\cdot 10^{-2}$} & \multirow{4}{*}{$5.793$} & \multirow{4}{*}{$1.4\cdot 10^{-3}$} & & \multirow{4}{*}{$3\cdot 10^{-4}$} \\
 & \small{($c_2=1$)} & \small{($c_2=0.01$)} & & & & $5.78304$ & \\
\cline{1-4}
\multirow{2}{*}{$0.05$} & $5.73$ & $5.95$ & \multirow{2}{*}{$3.8\cdot 10^{-2}$} & & & \small{($c_2=0.5286$)} & \\
 & \small{($c_2=1$)} & \small{($c_2=0.01$)} & & & & & \\
\hline
\multirow{2}{*}{$0.01$} & $5.775$ & $5.810$ & \multirow{2}{*}{$6\cdot 10^{-3}$} & \multirow{4}{*}{$5.785$} & \multirow{4}{*}{$3.1\cdot 10^{-5}$} & & \multirow{4}{*}{$3.5\cdot 10^{-6}$} \\
 & \small{($c_2=1$)} & \small{($c_2=0.01$)} & & & & $5.7848$ & \\
\cline{1-4}
\multirow{2}{*}{$0.005$} & $5.779$ & $5.800$ & \multirow{2}{*}{$3.6\cdot 10^{-3}$} & & & \small{($c_2=0.514$)} & \\
 & \small{($c_2=1$)} & \small{($c_2=0.01$)} & & & & & \\
\hline
\end{tabular}
\end{center}
\caption{Example 3. Numerical data of Figure \ref{figblowup3} (case 2: $m=2$, $c_1=0$). The blow-up time of the corresponding (exact) solution is assumed to be $\hat{t}\approx 5.78482$.}
\label{tabla3c2}
\end{table}

\begin{table}[h]
\begin{center}
\begin{tabular}{|c|c|c|c|c|c|c|}
  \hline
  \multirow{2}{*}{$h$} & \multicolumn{2}{|c|}{Case 1 - $c_1=0.5$} & \multicolumn{2}{|c|}{Case 2 - Radau I} & \multicolumn{2}{|c|}{Case 2 - $c_2=0.5$} \\
  \cline{2-7}
   & Blow-up & Rel. error & Blow-up & Rel. error & Blow-up & Rel. error \\
  \hline
  $0.1$ & $5.694865$ & $1.6\cdot 10^{-2}$ & $5.751797$ & $5.7\cdot 10^{-3}$ & $5.788215$ & $5.9\cdot 10^{-4}$ \\
  \hline
  $0.05$ & $5.739117$ & $7.9\cdot 10^{-3}$ & $5.767963$ & $2.9\cdot 10^{-3}$ & $5.786612$ & $3.1\cdot 10^{-4}$ \\
  \hline
  $0.01$ & $5.775121$ & $1.7\cdot 10^{-3}$ & $5.781037$ & $6.5\cdot 10^{-4}$ & $5.784995$ & $3\cdot 10^{-5}$ \\
  \hline
  $0.005$ & $5.780132$ & $8.1\cdot 10^{-4}$ & $5.783066$ & $3\cdot 10^{-4}$ & $5.784875$ & $9.5\cdot 10^{-6}$ \\
  \hline
\end{tabular}
\end{center}
\caption{Example 3. Numerical estimations and relative errors of the blow-up time of collocation solutions in case 1 with $c_1=0.5$, and case 2 with \textit{Radau I} collocation points, $c_1=0$, $c_2=2/3$, and $c_1=0$, $c_2=0.5$, for different initial stepsizes $h$. The blow-up time of the corresponding (exact) solution is assumed to be $\hat{t}\approx 5.78482$.}
\label{tabla3c105radau}
\end{table}

\begin{table}[h]
\begin{center}
\begin{tabular}{|c|c|c|c|c|c|c|c|}
\hline
Case 1 & \multicolumn{3}{|c|}{Varying $c_1$} & \multicolumn{2}{|c|}{Extrapolation} & \multicolumn{2}{|c|}{Intersection} \\
\hline
$h$ & min. $t$ & max. $t$ & Rel. error & $t$ & Rel. error & $t$ & Rel. error \\
\hline
\multirow{2}{*}{$0.01$} & $1.593$ & $1.722$ & \multirow{2}{*}{$7.8\cdot 10^{-2}$} & \multirow{4}{*}{$1.6436$} & \multirow{4}{*}{$1.4\cdot 10^{-3}$} & & \multirow{4}{*}{$5.2\cdot 10^{-4}$} \\
 & \small{($c_1=1$)} & \small{($c_1=0.05$)} & & & & $1.6467$ & \\
\cline{1-4}
\multirow{2}{*}{$0.005$} & $1.613$ & $1.691$ & \multirow{2}{*}{$4.7\cdot 10^{-2}$} & & & \small{($c_1=0.444$)} & \\
 & \small{($c_1=1$)} & \small{($c_1=0.05$)} & & & & & \\
\hline
\multirow{2}{*}{$0.001$} & $1.6367$ & $1.6571$ & \multirow{2}{*}{$1.2\cdot 10^{-2}$} & \multirow{4}{*}{$1.6456$} & \multirow{4}{*}{$1.2\cdot 10^{-4}$} & & \multirow{4}{*}{$6.6\cdot 10^{-5}$} \\
 & \small{($c_1=1$)} & \small{($c_1=0.05$)} & & & & $1.64595$ & \\
\cline{1-4}
\multirow{2}{*}{$0.0005$} & $1.6406$ & $1.6521$ & \multirow{2}{*}{$7\cdot 10^{-3}$} & & & \small{($c_1=0.461$)} & \\
 & \small{($c_1=1$)} & \small{($c_1=0.05$)} & & & & & \\
\hline
\end{tabular}
\end{center}
\caption{Example 4. Numerical data of Figure \ref{figblowup4} (case 1: $m=1$, $c_1>0$). The blow-up time of the corresponding (exact) solution is assumed to be $\hat{t}\approx 1.645842$.}
\label{tabla4c1}
\end{table}

\begin{table}[h]
\begin{center}
\begin{tabular}{|c|c|c|c|c|c|}
\hline
Case 2 & \multicolumn{3}{|c|}{Varying $c_2$} & \multicolumn{2}{|c|}{Extrapolation} \\
\hline
$h$ & min. $t$ & max. $t$ & Rel. error & $t$ & Rel. error \\
\hline
\multirow{2}{*}{$0.01$} & $1.6445$ & $1.6490$ & \multirow{2}{*}{$2.7\cdot 10^{-3}$} & \multirow{4}{*}{$1.64648$} & \multirow{4}{*}{$3.9\cdot 10^{-4}$} \\
 & \small{($c_2=0.95$)} & \small{($c_2=0.33$)} & & & \\
\cline{1-4}
\multirow{2}{*}{$0.005$} & $1.6456$ & $1.6476$ & \multirow{2}{*}{$1.2\cdot 10^{-3}$} & & \\
 & \small{($c_2=0.9$)} & \small{($c_2=0.3$)} & & & \\
\hline
\multirow{2}{*}{$0.001$} & $1.64577$ & $1.64622$ & \multirow{2}{*}{$2.7\cdot 10^{-4}$} & \multirow{4}{*}{$1.64586$} & \multirow{4}{*}{$1.4\cdot 10^{-5}$} \\
 & \small{($c_2=0.9$)} & \small{($c_2=0.3$)} & & & \\
\cline{1-4}
\multirow{2}{*}{$0.0005$} & $1.64581$ & $1.64607$ & \multirow{2}{*}{$1.6\cdot 10^{-4}$} & & \\
 & \small{($c_2=0.85$)} & \small{($c_2=0.28$)} & & & \\
\hline
\end{tabular}
\end{center}
\caption{Example 4. Numerical data of Figure \ref{figblowup4} (case 1: $m=2$, $c_1=0$). The intersection method can not be applied because there are more than one intersection. The blow-up time of the corresponding (exact) solution is assumed to be $\hat{t}\approx 1.645842$.}
\label{tabla4c2}
\end{table}

\begin{table}[h]
\begin{center}
\begin{tabular}{|c|c|c|c|c|}
  \hline
  \multirow{2}{*}{$h$} & \multicolumn{2}{|c|}{Case 1 - $c_1=0.5$} & \multicolumn{2}{|c|}{Case 2 - Radau I} \\
  \cline{2-5}
   & Blow-up & Rel. error & Blow-up & Rel. error \\
  \hline
  $0.01$ & $1.638700$ & $4.3\cdot 10^{-3}$ & $1.646991$ & $7\cdot 10^{-4}$ \\
  \hline
  $0.005$ & $1.642843$ & $1.8\cdot 10^{-3}$ & $1.646540$ & $4.2\cdot 10^{-4}$ \\
  \hline
  $0.001$ & $1.645172$ & $4.1\cdot 10^{-4}$ & $1.645985$ & $8.7\cdot 10^{-5}$ \\
  \hline
  $0.0005$ & $1.645491$ & $2.1\cdot 10^{-4}$ & $1.645919$ & $4.7\cdot 10^{-5}$ \\
  \hline
\end{tabular}
\end{center}
\caption{Example 4. Numerical estimations and relative errors of the blow-up time of collocation solutions in case 1 with $c_1=0.5$, and case 2 with \textit{Radau I} collocation points, $c_1=0$, $c_2=2/3$, for different initial stepsizes $h$. The blow-up time of the corresponding (exact) solution is assumed to be $\hat{t}\approx 1.645842$.}
\label{tabla4c105radau}
\end{table}


\begin{thebibliography}{00}

\bibitem{Kirk09} C. M. Kirk. Numerical and asymtotic analysis of a
  localized heat source undergoing periodic motion. \textit{Nonlinear
  Anal.} \textbf{71} (2009), e2168--e2172.

\bibitem{CalCapo09} F. Calabr\`o, G. Capobianco. Blowing up behavior for
  a class of nonlinear VIEs connected with parabolic
  PDEs. \textit{J. Comput. Appl. Math.} \textbf{228} (2009), 580--588.

\bibitem{Myd94} W. Mydlarczyk. A condition for finite blow-up time for
  a Volterra integral equation. \textit{J. Math. Anal. Appl.} \textbf{181}
  (1994), 248--253.

\bibitem{Myd99} W. Mydlarczyk. The blow-up solutions of integral
  equations. \textit{Colloq. Math.} \textbf{79} (1999), 147--156.

\bibitem{Okra07} T. Ma\l olepszy, W. Okrasi\'nski. Conditions for
  blow-up of solutions of some nonlinear Volterra integral
  equations. \textit{J. Comput. Appl. Math.} \textbf{205} (2007), 744--750.

\bibitem{Okra08} T. Ma\l olepszy, W. Okrasi\'nski. Blow-up conditions
  for nonlinear Volterra integral equations with power
  nonlinearity. \textit{Appl. Math. Letters.} \textbf{21} (2008), 307--312.

\bibitem{Okra10} T. Ma\l olepszy, W. Okrasi\'nski. Blow-up time for
  solutions for some nonlinear Volterra integral
  equations. \textit{J. Math. Anal. Appl.} \textbf{366} (2010),
  372--384.

\bibitem{Brun04} H. Brunner. Collocation Methods for Volterra Integral
  and Related Functional Differential Equations. Cambridge University
  Press, Cambridge (2004).

\bibitem{YB13} Z. W. Yang, H. Brunner. Blow-up behavior of collocation
  solutions to Hammerstein-type Volterra integral equations.
  \textit{SIAM J. Numer. Anal.} \textbf{51} (2013), 2260--2282.

\bibitem{Kras84} M. A. Krasnosel'skii, P. P. Zabreiko. Geometric
  Methods of Nonlinear Analysis. Springer Verlag, New York, 1984.

\bibitem{BB11} R. Ben\'{\i}tez, V. J. Bol\'os. Existence and uniqueness of
  nontrivial collocation solutions of implicitly linear homogeneous
  Volterra integral equations. \textit{J. Comput. Appl. Math.} \textbf{235}
  (2011), 3661--3672.

\bibitem{Brun92} H. Brunner. Implicitly linear collocation methods for
  nonlinear Volterra integral equations. \textit{Appl. Numer. Math.} \textbf{9}
  (1992), 235--247.

\bibitem{AB03} M.R. Arias, R. Ben\'{\i}tez, Aspects of the behaviour of
  solutions of nonlinear Abel equations. \textit{Nonlinear Anal. T.M.A.}, \textbf{54}
  (2003), pp. 1241--1249.

\bibitem{AB01} M.R. Arias, R. Ben\'{\i}tez, A note of the uniqueness and
  the attractive behaviour of solutions for nonlinear Volterra
  equations, \textit{J. Integral Equations Appl}. \textbf{13} (4)
  (2001) 305--310.



\end{thebibliography}
\end{document}